\documentclass[10pt]{amsart}
\usepackage{amsmath,amssymb,amsthm}
\usepackage{graphicx}
\usepackage{color}
\newtheorem{theorem}{Theorem}    

\newtheorem{proposition}{Proposition} 
\newtheorem{conjecture}{Conjecture} 
\newtheorem{question}{Question} 
\newtheorem{corollary}{Corollary}
\theoremstyle{definition}
\newtheorem{definition}[theorem]{Definition}

\newtheorem{remark}[theorem]{Remark}
\newtheorem*{remark*}{Remark}
\newcommand{\Z}{\mathbb{Z}}
\newcommand{\R}{\mathbb{R}}

\DeclareMathOperator{\rank}{rank}

\title{A slice Cromwell inequality of homogeneous links}
\author{Tetsuya Ito}

\begin{document}

\begin{abstract}
Cromwell proved that the minimum $v$-degree of the HOMFLY polynomial of homogeneous link $L$ is bounded above by $1-\chi(L)$, where $\chi(L)$ is the maximum Euler characteristic of Seifert surfaces of $L$.
We prove its slice version, stating that the minimum $v$-degree of the HOMFLY polynomial of homogeneous link $L$ is bounded above by $1-\chi_4(L)$, the maximum 4-dimensional Euler characteristic of $L$. As a byproduct, we prove a conjecture of Stoimenow that for an alternating link, the minimum $v$-degree of the HOMFLY polynomial is smaller than or equal to its signature. 
\end{abstract}

\maketitle

\section{Introduction}

A \emph{homogeneous link} is a link represented by a homogeneous diagram, a diagram whose Seifert graph is homogeneous. It is a common generalization of alternating links and positive or negative links, and has various nice properties.

Let $\chi(L)$ be the maximum Euler characteristic of Seifert surfaces of $L$ and let $P_L(v,z)$ be the HOMFLY polynomial of a link $L$, given by the skein relation
\[ v^{-1}P_{+}(v,z) - vP_{-}(v,z) = zP_{0}(v,z), \quad P_{\sf Unknot}(v,z)=1 \]
In \cite[Theorem 4 (b)]{cr} Cromwell showed the following, which we call the \emph{Cromwell inequality}.

\begin{theorem}[Cromwell inequality]
\label{theorem:Cromwell}
For a homogeneous link\footnote{In \cite{cr} link is always assumed to be non-split and the theorem is proved for non-split links, but one can check that the theorem applies for non-split links.} $L$, 
\begin{equation} 
\label{eqn:Cromwell-inequality-link}
\min \deg_v P_K(v,z) \leq 1-\chi(L)
\end{equation}
holds. Furthermore, the equality holds if and only if $L$ is positive.
\end{theorem}

For a link $L$ with homogeneous diagram $D$, Seifert's algorithm gives a maximum Euler characteristic Seifert surface of $L$ hence $\chi(L)=s(D)-c(D)$, where $s(D)$ is the number of Seifert circles and $c(D)$ is the number of crossings of $D$ \cite[Corollary 4.1]{cr}. Thus we may understand Theorem \ref{theorem:Cromwell} as the inequality 
\begin{equation}
\label{eqn:Cromwell-inequality-diagram}
\min \deg_v P_K(v,z) \leq -s(D) + c(D)+1
\end{equation}
 for a homogeneous diagram $D$ of $L$.

The aim of this note is to give a strengthened version of Cromwell's inequality.
We point out that Cromwell's argument, as we will state in Theorem \ref{theorem:Cromwell2}, actually proves a much stronger result than stated in Theorem \ref{theorem:Cromwell}. Based on this observation, we prove the following generalization of diagram version of the Cromwell inequality \eqref{eqn:Cromwell-inequality-diagram}.

\begin{theorem}
\label{theorem:main}
For a homogeneous link $L$ with homogeneous diagram $D$, 
\[ \min \deg_v P_K(v,z) \leq -s(D)+w(D)+2s_+(D)+1 -2\#_{sp} L\]
holds.
\end{theorem}

Here, $w(D)$ is the writhe of $D$, $s_+(D)$ is the number of connected components of diagram obtained from $D$ by resolving all the negative crossings of $D$, and $\#_{sp} L$ is the number of split components of $L$.

The right-hand side in Theorem \ref{theorem:main} appears in the following so-called sharper Bennequin-type inequality. Let $\chi_4(L)$ be the maximum Euler characteristic of a smoothly embedded surface in $B^{4}$ whose boundary is $L$. A \emph{slice-torus invariant} of knot is a homomorphism $\phi:\mathcal{C} \rightarrow \R$ from the smooth knot concordance group $\mathcal{C}$ to $\R$, such that $\phi(T_{p,q})= g_4(T_{p,q})$ for the $(p,q)$-torus knot $T_{p,q}$ \cite{le}. For the definition of slice-torus invariant of links we refer to \cite{cc}.

\begin{theorem}\cite[Theorem 1.4]{cc}
\label{theorem:slice-torus}
Let $\phi$ be a slice-torus invariant.
Then for a link $L$ and its non-splittable diagram $D$, 
\[ -s(D)+w(D)+2s_+(D)+1-2\#_{sp} L \leq 2\left(\phi(L) - \frac{\# L-1}{2} \right) \leq 1-\chi_4(L)\]
holds. 
\end{theorem}
Here $\#L$ denotes the number of components of $L$, and a diagram $D$ is \emph{non-splittable} means that the number of connected components of $D$ is equal to $\#_{sp} L$, the number of split components of $L$. 

Since homogeneous diagram is non-splittable \cite[Corollary 3]{cr}, Theorem \ref{theorem:main} and Theorem \ref{theorem:slice-torus} upgrades the Cromwell inequality \eqref{eqn:Cromwell-inequality-link} to the \emph{slice Cromwell inequality}.

\begin{corollary}[Slice Cromwell inequality]
\label{cor:slice}
For a homogeneous link $L$, 
\[ \min \deg_v P_L(v,z) \leq 1-\chi_4(L)\]
holds.
\end{corollary}

Recall that $\min \deg_{v} P_L(v,z)$ appears as the Morton-Franks-Williams inequality  \cite{fw,mo}
\[ \overline{sl}(L) \leq \min \deg_v P_L(v,z) - 1, \]
of the maximum self-linking number 
\[ \overline{sl}(L) = \max\{-s(D)+w(D)\: | \: D \mbox{ is a diagram of } L\}.\] Theorem \ref{theorem:main} says that for homogeneous links, the Morton-Franks-Williams inequality subsumes all the Bennequin-type inequality 
\[ \overline{sl}(L) \leq 2\left(\phi(L) - \frac{\# L-1}{2} \right) -1 \]
of the slice-torus invariants $\phi$.

As a generalization of the equality condition of Cromwell inequality, we raise the following conjectures. 

\begin{conjecture}
\label{conj:positive}
Let $L$ be a homogeneous link.
\begin{itemize}
\item[(i)] The slice Cromwell inequality is equality if and only if $L$ is positive.
\item[(ii)] For a homogeneous diagram $D$ of $L$, the inequality in Theorem \ref{theorem:main} is equality if and only if $D$ is a positive diagram.
\end{itemize}
\end{conjecture}

Conjecture \ref{conj:positive} is related to the following problem by Baader \cite{ba}.

\begin{question}
\label{ques:Baader}
Is homogeneous quasipositive link positive ?
\end{question}

Here a link $L$ is called \emph{quasipositive} if it is a closure of quasipositive braid, a product of conjugates of the positive generators $\sigma_{1},\ldots,\sigma_{n-1}$ of the braid group.

Although Question \ref{ques:Baader} remains open even for alternating links, in \cite{it}, extending alternating link case argument \cite{or}, we gave a partial affirmative answer. We showed that if a homogeneous link $L$ admits a homogeneous diagram that attains the maximum self-linking number (for example, when the number of Seifert circles of $D$ is equal to the braid index of $L$), then $L$ is quasipositive if and only if the diagram is positive.

Conjecture \ref{conj:positive} leads to an affirmative answer to Question \ref{ques:Baader}. 

\begin{theorem}
If Conjecture \ref{conj:positive} (i) or (ii) is true, quasipositive homogeneous links are positive.
\end{theorem}
\begin{proof}
For homogeneous link $L$ with homogeneous diagram $D$ we have 
\begin{align*}
\overline{sl}(L) 
&\leq \min \deg_{v} P_L(v,z)-1 \leq -s(D)+w(D) + 2s_+(D) - 2\#_{sp}L \leq -\chi_4(L)
\end{align*}
Since $-\chi_4(L)=\overline{sl}(L)$ holds for a quasipositive link $L$ \cite{rud}, it follows that a quasipositive homogeneous link attains the equality for both the slice Cromwell inequality and the equality of Theorem \ref{theorem:main}.
Thus assuming Conjecture \ref{conj:positive} (i) or (ii) we conclude that $L$ is positive. 
\end{proof}

Finally, we point out Theorem \ref{theorem:Cromwell2}, the actual content of Cromwell's proof, can be used to solve \cite[Problem 1.16]{Oht+}, \cite[Conjecture 4.1]{st}\footnote{\cite[Problem 1.16]{Oht+}, \cite[Conjecture 4.1]{st} also asks the inequality $\min \deg_a F_L(a^{-1},z) \leq \min \deg_v P_K(v,z)$ for the Kauffman polynomial $F_L(a,z)$. This inequality was proven in \cite[Corollary 5.4]{rut}.} concerning the signature\footnote{Here we use the convention that $\sigma(\mbox{Positive trefoil})=2$.} $\sigma(L)$ and the minimum degree of the HOMFLY polynomial for alternating links.

\begin{theorem}
\label{theorem:signature}
For an alternating link $L$, 
\[ \min \deg_v P_L(v,z) + \#_{sp}(L) -1 \leq \sigma(L)\]
holds.
\end{theorem}

\section*{Acknowledgements}
 
The author is partially supported by JSPS KAKENHI Grant Numbers 21H04428 and 23K03110.

\section{Proof of Theorem}

Let $G$ be a signed graph which is not necessarily connected, but is not allowed to have a self-loop. A signed means that the sign $+$ or $-$ is assigned for each edge. A signed graph is \emph{positive} (resp.\emph{negative}) if all the signs of edges are positive (resp. negative).

A vertex $v$ of a graph $G$ is a \emph{cut vertex} if removing $v$ increase the number of connected components. By splitting the graph along its cut vertices, each graph is decomposed as union of \emph{blocks}, connected maximum subgraphs without cut vertices.
The \emph{rank} of a graph $G$ is defined by $\rank(G)= -\#V(G)+\#E(G)+\#G$, where $V(G)$ and $E(G)$ denotes the set of vertices and edges of $G$, and $\#G$ denotes the number of connected components of $G$.

The \emph{Seifert graph} $G_D$ of a link diagram $D$ is a signed graph whose vertices are Seifert circles and whose edges are crossings. 

\begin{definition}[Homogeneous graphs and links]
A signed graph is \emph{homogeneous} if each block is either positive or negative. A link diagram $D$ is \emph{homogeneous} if its Seifert graph $G_D$ is homogeneous.  
\end{definition}

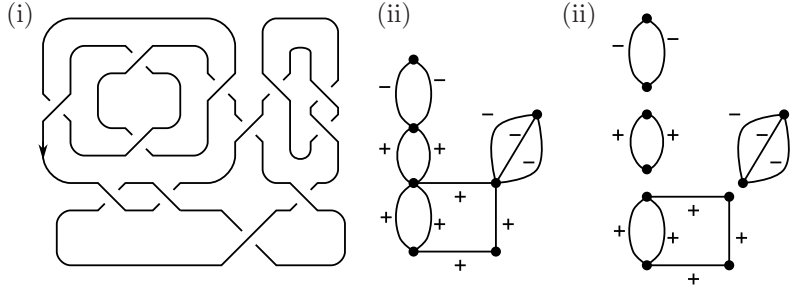
\begin{figure}[htbp]
\begin{center}
\includegraphics*[width=100mm]{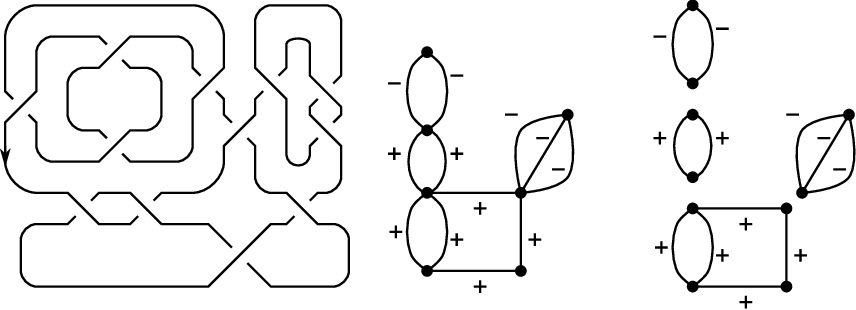}
\begin{picture}(0,0)
\put(-300,100) {(i)}
\put(-160,100) {(ii)}
\put(-90,100) {(ii)}
\end{picture}
\end{center}
\caption{(i) Homegenous diagram (ii) Seifert graph (iii) The blocks of Seifert graph} 
\label{fig:homogeneous}
\end{figure} 

To prove Theorem \ref{theorem:main}, we pin down what Cromwell actually proved in his proof of Cromwell inequality.

\begin{theorem}[Cromwell \cite{cr}]
\label{theorem:Cromwell2}
Let $D$ be a connected homogeneous diagram of a link $L$. Let $B_{i}$ $(i=1,\ldots,d)$ be the block of its Seifert graph $G_D$. Then $P_{L}(v,z)$ contains a monomial $v^{a}z^{b}$ with $a=\sum_{i=1}^{d}\varepsilon(B_i)\rank(B_i)$ and $b=\rank G_D(= 1-\chi(L))$. Here $\varepsilon(B_i)$ is the sign of the block $B_i$.
\end{theorem}
For reader's convenience we give a brief outline of the proof, emphasizing the key properties proved in the proof. 

\begin{proof}[Sketch of proof]
Starting from a given diagram $D$ of a link $L$, the HOMFLY polynomial can be computed by applying skein relation repeatedly.
Such a computation is summarized by so-called the \emph{skein (resolution) tree}, a rooted binary tree having the following properties.

\begin{itemize}
\item Each node is labelled by a link diagram and each edge is labelled by a monomial.
\item The root is labelled by the diagram $D$.
\item The triple $(D_{\sf parent},D_{\sf child 1},D_{\sf child 2})$ consisting of the diagram $D_{\sf parent}$ associated to a non-terminal nodes and the diagrams $D_{\sf child 1},D_{\sf child 2}$ associated to its child nodes forms a skein triple $(D_+,D_-,D_0)$ or $(D_{-},D_+,D_0)$.

In the former case, the edge connecting $D_{\sf parent}$ to $D_{\sf child 1}$ (resp. $D_{\sf parent}$ to $D_{\sf child 2}$) is labelled by the monomial $v^{2}$ (resp. $vz$). In the latter case, the edge connecting $D_{\sf parent}$ to $D_{\sf child 1}$ (resp. $D_{\sf parent}$ to $D_{\sf child 2}$) is labelled by the monomial $v^{-2}$ (resp. $-v^{-1}z$) (See Figure \ref{fig:skein-tree}).
\item A diagram associated a terminal node represents the unlink.
\end{itemize}

\begin{center}
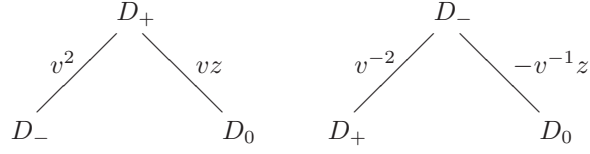
\begin{figure}[htbp]
\begin{picture}(220,55)
\put(40,50) {$D_+$}
\put(10,15){\line(1,1){30}}
\put(50,45){\line(1,-1){30}}
\put(0,5) {$D_-$}
\put(80,5) {$D_0$}
\put(15,30) {$v^{2}$}
\put(70,30) {$vz$}
\put(160,50) {$D_-$}
\put(130,15){\line(1,1){30}}
\put(170,45){\line(1,-1){30}}
\put(120,5) {$D_+$}
\put(200,5) {$D_0$}
\put(130,30) {$v^{-2}$}
\put(190,30) {$-v^{-1}z$}
\end{picture}
\label{fig:skein-tree}
\caption{Skein resolution tree}
\end{figure}
\end{center}

By the skein relation of the HOMFLY polynomial, the HOMFLY polynomial $P_L(v,z)$ is given by
\[ P_L(v,z) = \sum_{n}\pi(n) \left(\frac{v^{-1}-v}{z}\right)^{\#n-1} \]
where the summation runs all the terminal nodes $n$, 
$\pi(n)$ is the product of labels of edges that appear in the path from the terminal node $n$ to the root, and $\#n$ is the number of components of the unlink associated to the terminal node $n$.

Cromwell showed that in an appropriate skein tree coming from ascending diagram construction, if a terminal node $n$ satisfies $\max \deg_z \pi(n) \left(\frac{v^{-1}-v}{z}\right)^{\#n-1} = \rank G_D$ (i.e., if a terminal node $n$ contributes the highest $z$-degree terms of $P_L(v,z)$), it has the following properties.
\begin{itemize}
\item[(a)] The Seifert graph of a diagram associated to $n$ is a tree. Hence it represent the unknot and $\pi(n) \left(\frac{v^{-1}-v}{z}\right)^{\#n-1} = \pi(n)$ is a monomial.
\item[(b)] The sign of the monomial $\pi(n) \left(\frac{v^{-1}-v}{z}\right)^{\#n-1}$ does not depend on a choice of such a terminal node $n$. The sign is given by $\prod_{i=1}^{d} \varepsilon(B_i)^{\rank B_i}$. In particular, contributions of the terminal nodes never cancel.
\end{itemize}

Then the assertion comes from an observation that for the terminal node $n_0$ obtained by repeatedly resolving the crossings (i.e. the rightmost terminal node), its contribution is given by 
\[ \pi(n_0) \left(\frac{v^{-1}-v}{z}\right)^{\#n_0-1} = v^{\sum_{i=1}^{d}\varepsilon(B_i)\rank(B_i)}z^{\rank G_D} \]
\end{proof}

We rewrite $\sum_{i=1}^{d}\varepsilon(B_i)\rank(B_i)$ in terms of the quantities from diagram, rather than its Seifert graph.
\begin{proposition}
\label{prop:key}
For a connected homogeneous diagram $D$, $\sum_{i=1}\varepsilon(B_i)\rank(B_i) = -s(D)+w(D)+2s_+(D) - 1$.
\end{proposition}
\begin{proof}
Let $P$ and $N$ be the number of positive and negative blocks, and let $B^{+}_1,\ldots, B^{+}_P$ (resp. $B^{-}_1,\ldots, B^{-}_N$) be the positive (resp. negative) blocks of $G_D$.

In the procedure of the decomposition of Seifert graph $G_D$ into its blocks, cutting a graph at cut vertex increases the number of vertices by one. Therefore 
\[ s(D) + (P+N-1) = \sum_{i=1}^{P}\#V(B_i^{+}) + \sum_{i=1}^{N}\#V(B_i^{-}) \]
holds.
Similarly, $s_+(D)$ is the number of connected components of the graph obtained by deleting all the negative-signed edges hence 
\[ s_+(D) + (N-1) = \sum_{i=1}^{N}\#V(B_i^{-})\]
holds.
Thus
\begin{align*}
\sum_{i=1}^{d}\varepsilon(B_i)\rank(B_i) 
&= \sum_{i=1}^{P} \rank(B^{+}_i) - \sum_{i=1}^{N} \rank(B^{-}_i)\\
&= \sum_{i=1}^{P} (-\#V(B_i^{+})+\#E(B_i^{+})+1) - \sum_{i=1}^{N} (-\#V(B_i^{-})+\#E(B_i^{-})+1)\\
&= -\sum_{i=1}^{P}\#V(B_i^{+}) + \sum_{i=1}^{N}\#V(B_i^{-}) +w(D) +P-N\\
&= -s(D)-(P+N-1)  + 2\sum_{i=1}^{N}\#V(B_i^{-}) + w(D) + P-N \\
&= -s(D) + w(D) +2s_+(D) - 1
\end{align*}

\end{proof}

Then we prove the following assertion which is a bit stronger than Theorem \ref{theorem:main}.

\begin{theorem}
\label{theorem:main2}
For a homogeneous link $L$ with homogeneous diagram $D$, let $h_L(v) \in \Z[v,v^{-1}]$ be the highest $z$-degree term of the HOFMLY polynomial $P_L(v,z)$ of $L$, the coefficient of $z^{1-\chi(L)}$. Then $h_L(v)$ contains a monomial $v^{-s(D)+w(D)+2s_+(D)+1-2\#_{sp} L}$. In particular,
\[ \min \deg_v h_L(v) \leq -s(D)+w(D)+2s_+(D) +1 -2\#_{sp} L\]
\end{theorem}
\begin{proof}
The case $L$ is non-split follows from Theorem \ref{theorem:Cromwell2} and Proposition \ref{prop:key}.

For split link $L=L_1 \sqcup L_2 \sqcup \cdots \sqcup L_{\#_{sp} L}$, its homogeneous diagram $D$ is a disjoint union of homegeneous diagram $D_i$ of $L_i$ hence
\begin{align*}
\min \deg_v h_L(v)
&= \sum_{i=1}^{\#_{sp} L} \min \deg_v h_{L_i}(v) - (\#_{sp} L -1)\\
&\leq \sum_{i=1}^{\#_{sp} L} (-s(D_i) + w(D_i) +2s_+(D_i) - 1) - (\#_{sp} L -1)\\
& = -s(D)+w(D)+2s_+(D) +1 -2\#_{sp} L
\end{align*}
\end{proof}

\begin{remark}
\label{remark:positive}
The inequality 
\[\min \deg_v h_L(v) \leq -s(D)+w(D)+2s_+(D)+1-2\#_{sp} L \leq 1-\chi_4(L) \]
in Theorem \ref{theorem:main2} often attains the equality for non-positive diagrams,   even for alternating links.
For example, a non-positive alternating knot $6_1$ attains the equality. Thus, we cannot prove Conjecture \ref{conj:positive} (i) (ii) just by looking at the highest $z$-degree part $h_L(v)$ of the HOMFLY polynomial.
\end{remark}

The proof of Theorem \ref{theorem:signature} is similar to the proof of Theorem \ref{theorem:main}. We observe that $\sum_{i=1}^{d}\varepsilon(B_i)\rank(B_i)$ is nothing but the signature for alternating diagrams.

\begin{proof}[Proof of Theorem \ref{theorem:signature}]
For a connected, reduced alternating diagram $D$ of a link $L$, the signature is given by 
\[ \sigma(L)= w(D)- (d_+ - d_-) \]
where $d_{\pm}$ is the number of positive and negative edges of a spanning tree of the Seifert graph $G_D$ \cite[Theorem 2 (1)]{tr}\footnote{The formula in \cite[Theorem 2 (1)]{tr} contains an erroneous factor $\frac{1}{2}$.}.
Let $B^{+}_1,\ldots, B^{+}_P$ (resp. $B^{-}_1,\ldots, B^{-}_N$) be the positive (resp. negative) blocks of the Seifert graph $G_D$.

Since $d_+ = \sum_{i=1}^{P}(\#V(B_i^+) -1)$ and $d_{-} = \sum_{i=1}^{N}(\#V(B_i^-) -1) $
we conclude 
\begin{align*}
\sigma(L) &= w(D)- (d_+ - d_-)\\
&= \sum_{i=1}^{P}\#E(B_i^{+}) - \sum_{i=1}^{N} \#E(B_i^{-}) -  \sum_{i=1}^{P} (\#V(B_i^{+})-1) + \sum_{i=1}^{N} (\#V(B_i^{-})-1)\\
&= \sum_{i=1}^{P} (-\#V(B_i^{+})+\#E(B_i^{+})+1) - \sum_{i=1}^{N} (-\#V(B_i^{-})+\#E(B_i^{-})+1)\\
&= \sum_{i=1}^{d}\varepsilon(B_i)\rank(B_i)
\end{align*}
By Theorem \ref{theorem:Cromwell2} this shows the theorem for non-split case. Split  case easily follows from non-split case.
\end{proof}

\end{document}